\documentclass[12pt, reqno, twoside, letterpaper]{amsart}

\usepackage{url}
\usepackage[colorlinks,linkcolor=blue,anchorcolor=blue,
citecolor=blue,backref=page]{hyperref}
\usepackage{color}
\usepackage{graphics,epsfig}
\usepackage{graphicx}
\usepackage{float}

\hypersetup{breaklinks=true}
\usepackage{paperstyle}

\renewcommand*{\backref}[1]{}
\renewcommand*{\backrefalt}[4]{%
    \ifcase #1 (Not cited.)%
    \or        (p.\,#2)%
    \else      (pp.\,#2)%
    \fi}

\def\hMqx{\widehat M_q(x)}
\def\hSqx{\widehat S_q(x)}
\def\hDqx{\widehat D_q(x)}

\numberwithin{figure}{section}

\newcommand{\fq}{\ensuremath{\mathfrak{q}_\sharp}}

\title[M{\"o}bius function twisted by characters with powerful moduli]
      {Sums with the M{\"o}bius function twisted\\ by characters
       with powerful moduli}
      
\author[W.\ D.\ Banks]{William D.\ Banks}

\address{Department of Mathematics, 
         University of Missouri, 
         Columbia MO, USA.}

\email{bankswd@missouri.edu}
      
\author[I.\ E.\ Shparlinski]{Igor E.\ Shparlinski} 

\address{Department of Pure Mathematics,
		 University of New South Wales,
		 Sydney, NSW 2052, Australia.}
\email{igor.shparlinski@unsw.edu.au}
  
\date{\today}

\begin{document}

\begin{abstract}
In their recent work, the authors (2016) have combined classical ideas
of A.~G.~Postnikov (1956) and N.~M.~Korobov (1974) to derive improved bounds on
short character sums for certain nonprincipal characters with powerful
moduli. In the present paper, these results are used to bound sums
of the M{\"o}bius function twisted by characters of the same type,
complementing and improving some earlier work of B.~Green (2012).
To achieve this, we obtain a series of results about the size and 
zero-free region of $L$-functions with the same class of moduli.
\end{abstract}


\maketitle


\section{Introduction}
\label{sec:intro}

Our work in this paper is motivated, in part, by a
program of Sarnak~\cite{Sarn} to 
establish instances of a general pseudo-randomness principle related
to a famous conjecture of Chowla~\cite{Chowla}. Roughly speaking, the
principle asserts that the M\"obius function $\mu$ does not correlate with
any function $F$ of low complexity, so that
\begin{equation}
\label{eq:Sarnak}
\sum_{n\le x}\mu(n)F(n)=o\(\sum_{n\le x}|F(n)|\)\qquad(x\to\infty).
\end{equation}
Combining a result of Linial, Mansour and Nisan~\cite{LMN} with techniques of
Harman and K\'atai~\cite{HK}, Green~\cite{Green} has shown that
if $F:\{0,\ldots,N-1\}\to\{\pm 1\}$ has the property that
$F(n)$ can be computed from the binary digits of $n$ using a bounded depth circuit,
then $F$ is orthogonal to the M\"obius function $\mu$ in the sense that~\eqref{eq:Sarnak} holds; see~\cite[Theorem~1.1]{Green}.
Among other things, Green's proof~\cite{Green} relies on a bound
for a sum with the M{\"o}bius function twisted by a Dirichlet character $\chi$
of modulus $q=2^\gamma$. To formulate the result of~\cite{Green} we denote
\begin{equation}
\label{eq:M chi x defn} 
M(x,\chi)=\sum_{n\le x}\mu(n)\chi(n),
\end{equation}
where $\chi$ is a Dirichlet character modulo $q$;
we refer the reader to~\cite[Chapter~3]{IwKow} for the relevant
background on characters. We also denote
\begin{equation}
\label{eq:hatMqN-defn}
\hMqx=\max_{\chi\bmod q}|M(x,\chi)|,
\end{equation}
where the maximum is taken over all Dirichlet characters $\chi$ modulo $q$.
We remark that although the principal character $\chi_0$ is included in
the definition~\eqref{eq:hatMqN-defn}, the pure sum $M(x,\chi_0)$ 
plays no r\^ole as it satisfies a stronger bound than any bounds
currently known for $M(x,\chi)$,
$\chi\ne\chi_0$; see~\cite[Chapter~V, Section~5, Equation~(12)]{Wal}  (and also~\cite{Sound} for the best known 
bound under the Riemann Hypothesis).

According to~\cite[Theorem~4.1]{Green}, for some absolute constant $c_1>0$
and all moduli of the form
\begin{equation}
\label{eq:michelle}
q=2^\gamma\le e^{c_1\sqrt{\log x}},
\end{equation}
the bound
\begin{equation}
\label{eq:what-it-is}
\hMqx = O\(xe^{-c_1\sqrt{\log x}}\)
\end{equation}
holds, where the implied constant is absolute.

Our aim in the present paper is to improve this result 
in the three directions. Namely, we obtain a  new  bound which is
\begin{itemize}
\item  stronger than~\eqref{eq:what-it-is} and gives a better saving 
with a higher power of $\log x$ in the exponent;
\item  valid for a larger class of moduli $q$; 
\item nontrivial in a broader range of  the parameters  $q$ and $x$.
\end{itemize}
As immediate applications, our new bounds yield improvements of
some other results of Green~\cite{Green}. 

To achieve this aim, we derive a series of new results concerning the size and 
zero-free region of $L$-functions and their derivatives
with the same class of moduli, which we believe can be of
independent interest and have other applications. 

\section{Main  result}
\label{sec:results}

For any functions $f$ and $g$, the notations $f(x)\ll g(x)$, $g(x)\gg f(x)$ and
$f(x)=O(g(x))$ are used interchangeably to mean that
$|f(x)|\le c|g(x)|$ holds with some constant $c>0$.
Throughout the paper, we indicate explicitly
any parameters on which the implied constants may depend.

Given a natural number $q$, its \emph{core} (or kernel)
is the product $\fq$ over the prime divisors $p$ of $q$; that is,
$$
\fq=\prod_{p\mid q}p.
$$
In this paper, we are mainly interested in bounding the sums $M(x,\chi)$
for certain moduli $q$ that have a suitable
core $\fq$.  Specifically, we assume that
\begin{equation}
\label{eq:minmax}
\min\limits_{p\mid q}\{v_p(q)\}
\ge 0.7\gamma\qquad
\text{with}\quad\gamma=\max\limits_{p\mid q}\{v_p(q)\}\ge\gamma_0,
\end{equation}
where $\gamma_0>0$ is a sufficiently large absolute constant, and $v_p$ denotes 
the standard $p$-adic valuation (that is,  for $n \ne 0$ we have  $v_p(n)=\alpha$, where  
$\alpha$ is the largest integer such that $p^\alpha\mid n$). 
The technical condition~\eqref{eq:minmax} is needed in order to apply the results of
our earlier paper~\cite{BaSh}; it is likely that this constraint
can be modified and relaxed in various ways.

\begin{remark}
Our work in the present paper (and earlier in \cite{BaSh}) is motivated by
applications in the important special case that $q=p^\gamma$ is a prime power.
In this situation, the failure of the condition~\eqref{eq:minmax} simply means
that $q=O(1)$, and in this case there are usually results in the literature 
which are superior to ours. For example, we use this observation to 
derive Corollary~\ref{cor:Green-char-range} below.
\end{remark}

In addition to~\eqref{eq:M chi x defn} we consider sums of the form
\begin{equation}
\label{eq:psi tau chi x defn} 
\psi(x,\chi)=\sum_{n\le x}\Lambda(n)\chi(n),
\end{equation}
where $\Lambda$ is the von Mangoldt function;
such sums are classical objects of
study in analytic number theory (see, e.g.,~\cite[Section~5.9]{IwKow}
or~\cite[Section~11.3]{MontVau}).
We expect that the techniques of this paper can be applied to bound
other sums of number theoretic interest.

We treat the sums $M(x,\chi)$ and $\psi(x,\chi)$ in parallel.
Building on the results and techniques of~\cite{BaSh} we derive the
strongest bounds currently known for such sums when the
modulus $q$ of $\chi$ satisfies~\eqref{eq:minmax}.
In particular, we improve and  generalize Green's bound~\eqref{eq:what-it-is}.
Moreover, we obtain nontrivial bounds assuming only that
\begin{equation}
\label{eq:obama}
q\le\exp\(c_1(\log x)^{3/2}(\log\log x)^{-7/2}\)
\end{equation}
holds rather than~\eqref{eq:michelle} (that is, our results are nontrivial
for shorter sums).

\begin{theorem}
\label{thm:mu-Lambda-sums}
Let $q$ be a modulus satisfying~\eqref{eq:minmax}.
There is a constant $c>0$ that depends only on $\fq$ 
and has the following property.  Let
\begin{align*}
Q_1&=\exp\((\log q)^{7/3}(\log\log q)^{5/3}\),\\
Q_2&=\exp\((\log q)^7(\log\log q)^{-1}\),
\end{align*}
and for $j=1,2$ put
$$
E_j=\begin{cases}
\exp\(-c \log x \cdot (\log q)^{-2/3}(\log\log q)^{-4/3}\)\cdot(\log x)^j
&\quad\hbox{if $x\le Q_1$},\\
\exp\(-c(\log x\cdot\log q)^{1/2}(\log\log q)^{-1/2}\)
&\quad\hbox{if $Q_1<x\le Q_2$},\\
\exp\(-c(\log x)^{4/7}(\log\log x)^{-3/7}\)
&\quad\hbox{if $x>Q_2$}.
\end{cases}
$$
For every primitive character $\chi$ modulo $q$ we have
$$
\frac1xM(x,\chi)\ll E_1
\mand
\frac1x\psi(x,\chi)\ll E_2.
$$
\end{theorem}

We remark that 
Green's bound~\eqref{eq:what-it-is} applies to arbitrary
characters whereas Theorem~\ref{thm:mu-Lambda-sums}
is formulated only for primitive characters. However, in the
special case that $q=2^\gamma$, the conductor $q_0$ of an arbitrary
Dirichlet character $\chi$ modulo $q$ is a power of two (since $q_0\mid q$),
and $\chi$ is a primitive character modulo $q_0$.  When
$q_0\ge 2^{\gamma_0}$ we still apply Theorem~\ref{thm:mu-Lambda-sums}
with $q_0$ in place of $q$, and this only increases
the range~\eqref{eq:obama} in which we have a nontrivial bound. 

We now give a simpler formulation of Theorem~\ref{thm:mu-Lambda-sums} in the regime 
\begin{equation}
\label{eq:alpha}
q  = \exp\((\log x)^{\alpha + o(1)}\),
\end{equation}
where $\alpha>0$ is a fixed real parameter.

\begin{corollary}
\label{cor:alpha}
Let $\alpha>0$ be fixed,
let $q$ be a modulus satisfying~\eqref{eq:minmax},
and suppose that~\eqref{eq:alpha} holds.
For every primitive character $\chi$ modulo $q$ we have
$$
\max\{|M(x,\chi)|,|\psi(x,\chi)|\}\le x\exp\(-(\log x)^{\beta(\alpha)+o(1)}\),
$$
where 
$$
\beta(\alpha)=\begin{cases}
4/7 
&\quad\hbox{if $\alpha\le 1/7$},\\
(1+\alpha)/2
&\quad\hbox{if $1/7\le\alpha\le 3/7$},\\
1-2\alpha/3
&\quad\hbox{if $\alpha\ge 3/7$}.
\end{cases}
$$
\end{corollary}

Alternatively, we can define the function $\beta(\alpha)$ of Corollary~\ref{cor:alpha} by
$$
\beta(\alpha)=\min\left\{\max\left\{4/7,(1+\alpha)/2\right\},1-2\alpha/3\right\};
$$
see also Figure~\ref{fig:beta_alpha} for its behavior.
\begin{figure}[H]
  \centering
  \includegraphics[scale=0.36]{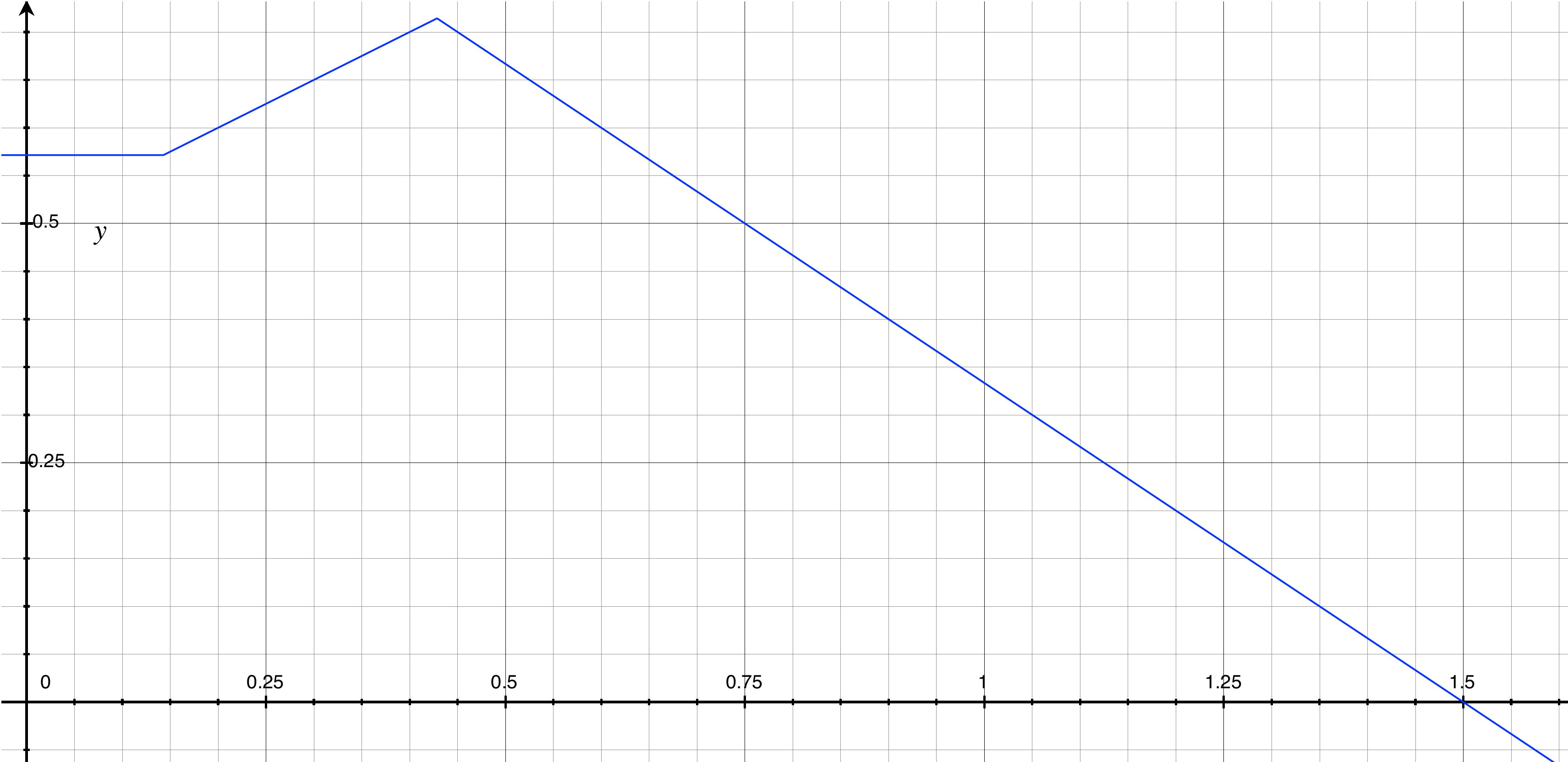}
  \caption{Function $\beta(\alpha)$}
  \label{fig:beta_alpha}
\end{figure}

Our proof of Theorem~\ref{thm:mu-Lambda-sums} is based on new bounds on 
Dirichlet polynomials and on the magnitudes and zero-free regions of certain
$L$-functions, which extend and improve those of Gallagher~\cite{Gal}, 
Iwaniec~\cite{Iwan} and the authors~\cite{BaSh}; see Section~\ref{sec:outline}
for more details. We believe these bounds are of independent interest
and may have many other arithmetic applications. In particular, 
one can apply these results to estimate character sums and exponential sums
twisted by the Liouville function, the divisor function, and other functions
of number theoretic interest.
 
\section{Applications}
\label{sec:appl}

It is easy to see that Theorem~\ref{thm:mu-Lambda-sums}
improves Green's result~\cite[Theorem~4.1]{Green}
in terms of both the range, increasing the right hand side in~\eqref{eq:michelle},
and the strength of the bound, reducing the right hand side in~\eqref{eq:what-it-is}. 
Although one can derive a general result with a trade-off between the range 
and the strength of the bound, we present only two special cases:
\begin{itemize}
\item[$(i)$] Improving the range~\eqref{eq:michelle} as much as possible
while preserving~\eqref{eq:what-it-is};
\item[$(ii)$] Improving the bound~\eqref{eq:what-it-is} as much as possible
(in this case, we are able to improve the range as well).
\end{itemize}
More precisely, for $(i)$ we note that the first bound of Theorem~\ref{thm:mu-Lambda-sums} 
implies that
\begin{equation}
\label{eq:thebigsick}
\log\(x^{-1}|M(x,\chi)|\)\ll\sqrt{\log x}
\end{equation}
holds for a character $\chi$ modulo $q=2^\gamma$ provided that
$\log q\ll(\log x)^{3/4}(\log\log x)^{-2}$ and the conductor of $\chi$
satisfies $q_0\ge 2^{\gamma_0}$ (so Theorem~\ref{thm:mu-Lambda-sums} applies).
Since the remaining characters $\chi$ modulo $q$ that have $q_0<2^{\gamma_0}$
comprise a \emph{finite} set, the bound~\eqref{eq:thebigsick} can be
achieved for those characters on a one-to-one basis; see, e.g., the paper
of Hinz~\cite{Hinz} and references therein.  Using~\eqref{eq:thebigsick}
to bound $\hMqx$ we derive the following statement.

\begin{corollary}
\label{cor:Green-char-range}
For some absolute constant $c_1>0$ and all moduli of the form
$$
q=2^\gamma\le e^{c_1(\log x)^{3/4} (\log\log x)^{-2}}\qquad(\gamma\in\NN)
$$
the bound~\eqref{eq:what-it-is} holds.
\end{corollary}

Similarly, for $(ii)$ we take $q$ smaller than in Corollary~\ref{cor:Green-char-range}
so the last two bounds in definition of $E_1$ in Theorem~\ref{thm:mu-Lambda-sums} 
dominate.  We obtain the following statement.

\begin{corollary}
\label{cor:Green-char-bound}
For some absolute constant $c_1>0$ and all moduli of the form
$$
q=2^\gamma\le e^{c_1(\log x)^{9/14}(\log\log x)^{-19/14}}\qquad(\gamma\in\NN)
$$
the bound 
$$
\hMqx\ll xe^{-c_1(\log x)^{4/7}(\log\log x)^{-3/7}}
$$
holds.
\end{corollary}

Notice that the exponents $3/4$ and $9/14$ that occur
in Corollaries~\ref{cor:Green-char-range} and~\ref{cor:Green-char-bound}
are the maximal roots $\alpha$ of the equations $\beta(\alpha)=1/2$
and $\beta(\alpha)=4/7$, respectively (in Figure~\ref{fig:beta_alpha}
these roots occur at the rightmost intersection points between the
graph of $\beta(\alpha)$ and the horizontal lines at height $1/2$ and $4/7$,
respectively).

We also remark that Corollaries~\ref{cor:Green-char-range}
and~\ref{cor:Green-char-bound} hold with $2^\gamma$ replaced by
$p^\gamma$ for any fixed prime $p$; in this case,
the constant $c_1>0$ can depend on $p$ but is otherwise absolute.

Next, we turn our attention to the following
exponential sums twisted by the M{\"o}bius function:
$$
S_q(x,a)=\sum_{n\le x}\mu(n)\exp(2 \pi i a n/q)\qquad(a\in\ZZ).
$$
We denote
$$
\hSqx=\max_{a \in \ZZ} \left|S_q(x,a)\right|.
$$ 
Green has shown (see~\cite[Corollary~4.2]{Green}) that
for some absolute constant $c_2>0$ and all moduli of the form
$$
q=2^\gamma\le e^{c_2\sqrt{\log x}},
$$
the bound
$$
\hSqx\ll xe^{-c_2\sqrt{\log x}}
$$
holds. The proof is obtained by relating $\hSqx$ to $\hMqx$
and then exploiting the bound~\eqref{eq:what-it-is}.
In order to get a nontrivial bound on $\hSqx$ in this way, the savings 
in the bound on  $\hMqx$ has to be a little larger than $q$. For this
reason, even though Corollaries~\ref{cor:Green-char-range}
and~\ref{cor:Green-char-bound} are suitable for such applications
(and lead to an improvement of~\cite[Corollary~4.2]{Green} in both the
range and strength of the bound), we need to slightly reduce the range of $q$. 
Furthermore, a full analogue of these results also holds for sums 
with the M{\"o}bius function in arithmetic progressions; that is,
we can bound the quantity 
$$
\hDqx=\max_{\substack{a \in \ZZ\\ \gcd(a,q)=1}} \left|D_q(x,a)\right|, 
$$ 
where 
$$
D_q(x,a)=\sum_{\substack{n\le x\\ n \equiv a \bmod q}}\mu(n). 
$$

\begin{corollary}
\label{cor:Green-exp-progr}
For some absolute constant $c_2>0$ and all moduli of the form
$$
q=2^\gamma \le e^{c_2 (\log x)^{3/5} (\log\log x)^{-4/5}}\qquad(\gamma\in\NN)
$$
we have
$$
\max\{\hSqx,\hDqx\}\ll  x  e^{-c_2 (\log x)^{4/7} (\log\log x)^{-3/7}}.
$$
\end{corollary}

Note that the exponent $3/5$ in Corollary~\ref{cor:Green-exp-progr}
is the root $\alpha$ of the equation $\beta(\alpha)=\alpha$,
and we have
$$
\min\left\{\beta(\alpha):\alpha\in(0, 3/5]\right\} = 4/7.
$$
As with the previous results, Corollary~\ref{cor:Green-exp-progr} also holds
with $2^\gamma$ replaced by $p^\gamma$ for any fixed prime $p$.

Finally, we mention that the results of the present paper
can be used to estimate \emph{Fourier-Walsh coefficients} 
$\widehat\mu_n(\cA)$ associated with the M{\"o}bius function,
interpolating between the result of Green~\cite[Proposition 1.2]{Green}
and that of Bourgain~\cite[Theorem~1]{Bour}.
Recall that for a natural number $n$ and a set $\cA\in \{0,\ldots,n-1\}$,
the coefficient $\widehat\mu_n(\cA)$ is defined by
$$
\widehat\mu_n(\cA) = \sum_{(x_0, \ldots, x_{n-1}) \in \{0,1\}^n} \prod_{j\in \cA} (1-2 x_j) \mu\(\sum_{j=0}^{n-1} 2^j x_j\).
$$
Green~\cite[Proposition 1.2]{Green} uses~\eqref{eq:what-it-is} to give a
nontrivial bound on $\widehat\mu_n(\cA)$ for sets $\cA$ of cardinality
$|\cA|\ll n^{1/2}/\log n$, and this bound on the cardinality is essential
the entire approach of~\cite{Green}. Bourgain~\cite[Theorem~1]{Bour}
gives a nontrivial bound on $\widehat\mu_n(\cA)$ for arbitrary sets.
Using Theorem~\ref{thm:mu-Lambda-sums} in the argument of Green~\cite{Green} one can 
obtain a result that is stronger than \cite[Proposition 1.2]{Green} but which
also holds for larger sets.

\section{Outline of the proof}
\label{sec:outline}

To prove Theorem~\ref{thm:mu-Lambda-sums},
we begin by extending the bound of~\cite[Theorem~2.2]{BaSh} to
cover Dirichlet polynomials supported on shorter intervals;
see Theorem~\ref{thm:T_chi(M,N;t)-bound}.  Adapting various ideas
and tools from~\cite{BaSh} to exploit this result on Dirichlet polynomials,
we give a strong bound on the size of the relevant $L$-functions $L(s,\chi)$
in the case that $\Im(s)$ is not too large; see Theorem~\ref{thm:bound-on-LFunction}.
For larger values of $\Im(s)$, we apply a well known bound of Iwaniec~\cite[Lemma~8]{Iwan}.

Having the upper bounds on these $L$-functions at our disposal,
we combine them with certain results and ideas of~\cite{BaSh} and~\cite{Iwan}
to obtain a new zero-free region, which is wider than all previously known regions;
see Theorem~\ref{thm:nonvanish-LFunction}.

Next, we introduce and apply an extension of a result of Montgomery and Vaughan
\cite[Lemma~6.3]{MontVau} which bounds the logarithmic derivative of a complex
function in terms of its zeros; see Lemma~\ref{lem:Brady1}.  
Similarly to~\cite{MontVau}, we use Lemma~\ref{lem:Brady1}
as a device that allows us to consolidate old and new bounds on the size 
and the zero-free region of our $L$-functions $L(s,\chi)$, and doing so leads
us to Theorem~\ref{eq:curious-beasts}, which provides strong bounds on
these $L$-functions, their logarithmic derivatives, and their reciprocals.

Finally, to conclude the proof of Theorem~\ref{thm:mu-Lambda-sums},
in Section~\ref{sec:proof main} we relate (via Perron's formula)
the sums $M(x,\chi)$ and $\psi(x,\chi)$ to the bounds given in
Theorem~\ref{eq:curious-beasts}.  This connection involves the use of a parameter $T$,
which we optimise to achieve the desired result.  It is worth mentioning that the
final optimisation step is somewhat delicate since our bounds behave very differently
in different ranges (e.g., see~\eqref{eq:Theta}).
 
\section{Bounds on Dirichlet polynomials}
\label{sec:bounds-Dirichlet-polys}

In this section, we study Dirichlet polynomials of the form
$$
T_\chi(M,N;t)=\sum_{M<n\le M+N}\chi(n)n^{it}\qquad(t\in\RR).
$$
As in~\cite{BaSh}, to bound these polynomials we approximate
each $T_\chi(M,N;t)$ with a sum of the form
$$
\sum_{M<n\le M+N}\chi(n)\e(G(n)),
$$
where $G$ is a polynomial with real coefficients, and
$\e(t)=e^{2\pi it}$ for all $t\in\RR$. Our result is the following statement,
which is more flexible than~\cite[Theorem~2.2]{BaSh}. 

\begin{theorem}
\label{thm:T_chi(M,N;t)-bound}
For every $C>1$ there are effectively computable constants
$\gamma_0,\xi_0>0$ depending only on $C$ such that the following property holds.
For any modulus $q$ satisfying~\eqref{eq:minmax}
and any primitive character $\chi$ modulo $q$, the bound
\begin{equation}
\label{eq:mainboundT}
T_\chi(M,N;t) 
\ll N^{1-\xi_0/\varrho^2}
\end{equation}
holds uniformly for all $M,N,t\in\RR$ subject to the conditions
\begin{equation}
\label{eq:condsT}
M\ge N,\qquad
\fq^{\gamma_0}\le N\le q^C
\mand|t|\le q^{\frac14(\log M)/\log\fq},
\end{equation}
where $\varrho=(\log q)/\log N$
and  implied constant in~\eqref{eq:mainboundT} depends only on $C$.
\end{theorem}

\begin{proof}
Let $\gamma_0$ be a positive constant exceeding $e^{200}$, and put
$$
\gamma_1=\frac{\log N}{\log\fq}\mand
\eps=\frac{5}{4\gamma_1}=\frac{5\log\fq}{4\log N}.
$$
Using~\eqref{eq:minmax} and~\eqref{eq:condsT} we have
$$
\log q=\sum_{p\mid q}v_p(q)\log p\le\gamma\sum_{p\mid q}\log p=\gamma\log \fq
\le\gamma\gamma_1^{-1}\log N,
$$
hence $\varrho$ lies in $[C^{-1},\gamma/\gamma_1]$.
Let $s=\fl{\eps\gamma/\varrho}$.
Since $\eps\gamma/\varrho\ge\eps\gamma_1=\frac54$ we have
\begin{equation}
\label{eq:halfepsgamma}
\tfrac15\eps\gamma/\varrho\le
\eps\gamma/\varrho-1<s\le \eps\gamma/\varrho,
\end{equation}
and therefore $s\asymp\gamma/\varrho$.

Now let $\nu=\rf{\gamma/(3s)}$.
For any real number $x$, we have the estimate
$$
(1+x)^{it} = \e(tG(x))\(1+O(|t||x|^{\nu})\), 
$$
where $G$ is the polynomial given by
$$
G(x)=\frac{1}{2\pi}\sum_{r=1}^{\nu-1}(-1)^{r-1}\frac{x^r}{r}.
$$
Hence, uniformly for $n\in[M+1,M+N]$ and $y,z\in [1,\fq^s]$,
taking into account that $M\ge N$, we have
\begin{equation}
\label{eq:G-approx}
\(n+\fq^syz\)^{it}=n^{it}(1+\fq^syz/n)^{it}
=n^{it}\e(tG(\fq^syz/n))+O(M^{-\nu}|t|\fq^{3s\nu}).
\end{equation}
Let $\cN$ be the set of integers coprime to $q$ in the interval
$[M+1,M+N]$.  Shifting the interval $[M+1,M+N]$ by the 
amount $\fq^s yz$, where $1\le y,z\le\fq^s$, we have the uniform estimate
\begin{align*}
T_\chi(M,N;t)
&=\sum_{n\in\cN}\chi(n)n^{it}\\
&=\sum_{n\in\cN}\chi(n+\fq^s yz)(n+\fq^s yz)^{it}+O(\fq^{3s}).
\end{align*}
Using~\eqref{eq:G-approx} and averaging over all such $y$ and $z$, 
it follows that
$$
T_\chi(M,N;t)=\fq^{-2s}V+O(\fq^{3s}+NM^{-\nu}|t|\fq^{3s\nu}),
$$
where
$$
V=\sum_{n\in\cN}\chi(n)n^{it}\sum_{y,z=1}^{\fq^s}\chi(1+\fq^s \overline nyz)
\e(tG(\fq^syz/n)).
$$
In this expression, we have used $\overline n$ to denote an integer
such that $n\overline n\equiv 1\bmod q$.
Since $\deg G\le\gamma/(3s)$, we can proceed in a manner that is identical to
the proof of~\cite[Theorem~2.1]{BaSh} to derive the bound
\begin{equation}
\label{eq:St almost}
T_\chi(M,N;t)\ll N^{1-\xi_0/\varrho^2}+\fq^{3s}+NM^{-\nu}|t|\fq^{3s\nu}
\end{equation}
in place of~\cite[Equation (5.16)]{BaSh}.
Below, we show that $\fq^{3s}\le N^{5\eps}$, hence the term $\fq^{3s}$
can be dropped from~\eqref{eq:St almost} if one makes
suitable initial choices of $\gamma_0$ and $\xi_0$.

To finish the proof, we need to bound the last term in~\eqref{eq:St almost}.
Let $\vartheta$ be such that $N^\vartheta=|t|+3$.
Since $\nu=\rf{\gamma/(3s)}$, it follows that $3s\nu\le\gamma+3s$,
and by~\eqref{eq:halfepsgamma} we have
$\nu\ge\gamma/(3s)\ge\varrho/(3\eps)$; thus, setting $\kappa=(\log M)/\log N$
it follows that
$$
NM^{-\nu}|t|\fq^{3s\nu}
\ll N^{1-\kappa\varrho/(3\eps)+\vartheta}\fq^{\gamma+3s}.
$$
In view of~\eqref{eq:minmax} the relation $q=\fq^{\gamma\mu}$
holds for some $\mu\in[0.7,1]$,
which implies that $\fq^\gamma\le N^{2\varrho}$ and
(using~\eqref{eq:halfepsgamma} again) that $\fq^{3s}\le N^{5\eps}$ (as claimed above).
Hence, 
$$
NM^{-\nu}|t|\fq^{3s\nu}
\ll N^{1-\kappa\varrho/(3\eps)+\vartheta+2\varrho+5\eps}.
$$
Inserting this bound into~\eqref{eq:St almost} and recalling that
$\eps=1.25\gamma_1^{-1}$, we see that~\eqref{eq:mainboundT}
a consequence of the inequality
$$
\vartheta\le\varrho(\kappa\gamma_1/(3.75)-2)
-\xi_0/\varrho^2-6.25\gamma_1^{-1}.
$$
Since $\varrho\ge C^{-1}$ and $\gamma_1\ge\gamma_0$, this condition
is met if $\gamma_0$ is sufficiently large and $\xi_0$ is sufficiently small
since the last inequality in~\eqref{eq:condsT}
implies that $\vartheta\le \frac14\varrho\kappa\gamma_1+o(1)$
as $N\to\infty$. This completes the proof.
\end{proof}

\begin{corollary}
\label{cor:S(M)-dyadic-bound}
Fix $C>1$, and let $\gamma_0,\xi_0>0$ have the property
described in Theorem~\ref{thm:T_chi(M,N;t)-bound}.
Let $q$ be a modulus satisfying~\eqref{eq:minmax} and let 
$\chi$ be a primitive character modulo~$q$. Put
$$
\tau=|t|+3,\qquad
\ell=\log(q\tau)\mand
Q_0=\max\left\{\fq^{\gamma_0},\fq^{4\ell/\log q}\right\}.
$$
Then, uniformly for $s=\sigma+it\in\CC$ and $M\ge Q_0$ the bound
$$
\sum_{M<n\le 2M}\chi(n)n^{-s} 
\ll \begin{cases}
M^{1-\sigma-\xi_0(\log M)^2/(\log q)^2}+M^{-\sigma-1}Q_0^2
&\quad\hbox{if $\fq^{\gamma_0}\le M\le q^C$},\\
M^{1-\sigma}q^{-c_0}+M^{-\sigma-1}Q_0^2
&\quad\hbox{if $M>q^C$}.
\end{cases}
$$
holds with
\begin{equation}
\label{eq:c_0-description}
c_0=C(C-1)^2\xi_0
\end{equation}
and an implied constant which depends only on $C$.
\end{corollary}

\begin{proof}
For simplicity we denote
$$
U_\chi(M)=\sum_{M<n\le 2M}\chi(n)n^{-s}
$$
and
$$
V_\chi(u)=T_\chi(M,u;-t)=\sum_{M<n\le M+u}\chi(n)n^{-it}
\qquad(0<u\le M).
$$
By partial summation it follows that
$$
U_\chi(M)=(2M)^{-\sigma}V_\chi(M)
+\sigma\int_0^M(u+M)^{-\sigma-1}V_\chi(u)\,du.
$$
Using the trivial bound $|V_\chi(u)|\le u+1$
for $0<u\le Q_0$ we see that
\begin{equation}
\label{eq:library1}
U_\chi(M)\ll M^{-\sigma}|V_\chi(M)|
+M^{-\sigma-1}Q_0^2
+M^{-\sigma-1}\int_{Q_0}^M|V_\chi(u)|\,du.
\end{equation}

We claim that the bound
\begin{equation}
\label{eq:panera}
V_\chi(u)\ll f(u)\qquad(u\ge Q_0)
\end{equation}
holds, where
$$
f(u)=\begin{cases}
u^{1-\xi_0(\log u)^2/(\log q)^2}
&\quad\hbox{if $\fq^{\gamma_0}\le u\le q^C$},\\
uq^{-c_0}
&\quad\hbox{if $u>q^C$}.
\end{cases}
$$
Indeed, let $u>Q_0$ be fixed.
If $u\in[\fq^{\gamma_0},q^C]$ then the bound~\eqref{eq:panera}
follows immediately from Theorem~\ref{thm:T_chi(M,N;t)-bound}
(note that our choice of $Q_0$ and the 
inequality $u>Q_0$ guarantee that the condition~\eqref{eq:condsT}
of Theorem~\ref{thm:T_chi(M,N;t)-bound} is met).
To bound $V_\chi(u)$ when $u>q^C$, we set
$$
J=\rf{\frac{u}{q^C}}\ge 2\mand
N=\frac{u}{J}\in(\tfrac12 q^C,q^C].
$$
Let $M_j=M+(j-1)N$ and put
$$
\cI_j=(M_j,M_{j+1}]=(M_j,M_j+N]\qquad(j\le J).
$$
Since $M_{J+1}=M+u$, we see that the interval $(M,M+u]$ is
a disjoint union of the intervals $\cI_j$, and thus
$$
V_\chi(u)=\sum_{j\le J}\sum_{n\in\cI_j}\chi(n)n^{-it}
=\sum_{j\le J}T_\chi(M_j,N;-t)\ll JN^{1-\xi_0(\log N)^2/(\log q)^2},
$$
where we have applied Theorem~\ref{thm:T_chi(M,N;t)-bound} in the last step.
Noting that $JN=u$ and $N>\tfrac12q^C\ge q^{C-1}$,
and taking into account~\eqref{eq:c_0-description},
we obtain~\eqref{eq:panera} for $u>q^{C}$.

Assuming as we may that $\xi_0<(3C^2)^{-1}$ it is easy to 
verify that $f(u)$ is an increasing function of $u$.
Hence, using~\eqref{eq:panera} (in the crude form $V_\chi(u)\ll f(M)$
for all $u \in [Q_0,M]$)
to bound the right side of~\eqref{eq:library1}, we derive that
$$
U_\chi(M)\ll M^{-\sigma}f(M)+M^{-\sigma-1}Q_0^2,
$$
and the result follows.
\end{proof}

\section{Bounds on $L$-functions}
\label{sec:bounds-on-L-functions-I}

We continue to use the notation of Section~\ref{sec:bounds-Dirichlet-polys}.
More specifically, let $\gamma_0,\xi_0>0$ have the property
described in Theorem~\ref{thm:T_chi(M,N;t)-bound} with $C=1001$ (say).
Let $q$ be a modulus satisfying~\eqref{eq:minmax},
$\chi$ a primitive character modulo~$q$, and put
\begin{equation}
\label{eq:def tau l M}
\tau=|t|+3,\qquad
\ell=\log(q\tau)\mand
Q_0=\max\left\{\fq^{\gamma_0},\fq^{4\ell/\log q}\right\}.
\end{equation}
Note that we can replace $\gamma_0$ with a larger value or replace
$\xi_0$ with a smaller value without changing the validity of
Theorem~\ref{thm:T_chi(M,N;t)-bound}.  Following Iwaniec~\cite{Iwan}
we put
\begin{equation}
\label{eq:def Y}
Y=\exp\(60(\ell\log 2\ell)^{3/4}\).
\end{equation}

We now present the following extension of~\cite[Lemma~6.1]{BaSh}. 

\begin{lemma}
\label{lem:etaXY-LFunction}
Suppose that the parameters $X$ and $\eta$ satisfy
\begin{equation}
\label{eq:Q_0Xeta-conditions}
X\ge Q_0,\qquad
\eta\in(0,\tfrac13)\mand
\eta\le\frac{\xi_0(\log X)^2}{(\log q)^2}-\frac{2\log\ell}{\log X}.
\end{equation}
Then for any $s=\sigma+it$ with $\sigma>1-\eta$ and    any primitive character $\chi$ modulo $q$we have
\begin{equation}
\label{eq:Kaldi}
L(s,\chi)\ll\begin{cases}
\eta^{-1}X^\eta&\quad\hbox{if $Y\le q^{1001}$},\\
\eta^{-1}X^\eta+Y^\eta \ell q^{-c_0}&\quad\hbox{if $Y>q^{1001}$},
\end{cases}
\end{equation}
where $c_0=10^9\xi_0$ and the implied constant in~\eqref{eq:Kaldi} is absolute.
\end{lemma}

\begin{proof}
From the proof of~\cite[Lemma~8]{Iwan} we see that the bounds
\begin{equation}
\label{eq:CopyPaper1}
\left|\sum_{n>Y}\chi(n)n^{-s}\right|\ll\eta^{-1}\fq^{100\eta}
\end{equation}
and
$$
\left|\sum_{n\le Y}\chi(n)n^{-s}\right|\ll\eta^{-1}Y^\eta
$$
hold for $\sigma>1-\eta$. In particular,~\eqref{eq:Kaldi} follows immediately
when $2X>Y$.  From now on we assume that $2X\le Y$.

Next, let $L$ be the unique integer such that the quantity
$X_*=2^{-L}Y$ lies in the interval $[X,2X)$; note that
$L\ge 1$ since $Y\ge 2X$. Let $\cS$ be the
collection of numbers of the form $M=2^{j-1}X_*$ with $j=1,\ldots,L$,
and notice that the interval $(X_*,Y]$ splits into $L$ disjoint subintervals
of the form $(M,2M]$ with $M\in\cS$, so that
$$
\sum_{X_*<n\le Y}\chi(n)n^{-s}=\sum_{M\in\cS}\sum_{M<n\le 2M}\chi(n)n^{-s}.
$$
We apply Corollary~\ref{cor:S(M)-dyadic-bound} with $C = 1001$. Hence, by~\eqref{eq:c_0-description}
we can even take a slightly larger value of $c_0$ than $c_0=10^9\xi_0$.

Now, let $\cS_1$ and $\cS_2$ be the (potentially empty)
sets of numbers $M\in\cS$ for which $M\le q^{1001}$ and $M>q^{1001}$, respectively.
Since $M\ge Q_0$ for every $M\in\cS$ by~\eqref{eq:Q_0Xeta-conditions},
we have by Corollary~\ref{cor:S(M)-dyadic-bound}:
\begin{align*}
\sum_{M\in\cS_1}\sum_{M<n\le 2M}\chi(n)n^{-s}
&\ll\sum_{M\in\cS_1}\(M^{\eta-\xi_0(\log M)^2/(\log q)^2}+M^{\eta-2}Q_0^2\),\\
\sum_{M\in\cS_2}\sum_{M<n\le 2M}\chi(n)n^{-s}
&\ll\sum_{M\in\cS_2}\(M^\eta q^{-c_0}+M^{\eta-2}Q_0^2\).
\end{align*}
Using~\eqref{eq:Q_0Xeta-conditions} and the fact that $|\cS_1| \le |\cS|=L\ll\ell$ we have
$$
\sum_{M\in\cS_1} M^{\eta-\xi_0(\log M)^2/(\log q)^2}
\le \sum_{M\in\cS_1} M^{2\log\ell/\log X}
\ll\sum_{M\in\cS_1}\ell^{-2}\ll 1.
$$
If $Y\le q^{1001}$, then $\cS_2=\varnothing$; for larger values of $Y$
we use the bound
$$
\sum_{M\in\cS_2} M^\eta q^{-c_0}\ll Y^\eta \ell q^{-c_0}.
$$
Putting the preceding bounds together, and taking into account that
$$
\sum_{M\in\cS}M^{\eta-2}Q_0^2 
\le Q_0^2 \sum_{j=1}^L(2^{j-1}X_*)^{\eta-2}  \ll  Q_0^2 X_*^{\eta-2}
\ll X^\eta,
$$
we derive that
\begin{equation}
\label{eq:CopyPaper2}
\sum_{X_*<n\le Y}\chi(n)n^{-s}\ll\begin{cases}
X^\eta&\quad\hbox{if $Y\le q^{1001}$},\\
X^\eta+Y^\eta \ell q^{-c_0}&\quad\hbox{if $Y>q^{1001}$}.
\end{cases}
\end{equation}

To finish the proof, we observe that
\begin{equation}
\label{eq:CopyPaper3}
\left|\sum_{n\le X_*}\chi(n)n^{-s}\right|
\le \sum_{n\le X_*}n^{\eta-1}
\le 1+\eta^{-1}(X_*^\eta-1)
\ll \eta^{-1}X^\eta.
\end{equation}
The result follows by combining the bounds~\eqref{eq:CopyPaper1}, 
\eqref{eq:CopyPaper2} and~\eqref{eq:CopyPaper3}.
\end{proof}

\begin{theorem}
\label{thm:bound-on-LFunction}
There are constants $A, B>0$ that depend only on $\fq$ and have
the following property. Put
$$
\eta_1=\frac{A}{(\log q)^{2/3}(\log\log q)^{1/3}},\qquad
\eta_2=\frac{A\log q}{\ell},
$$
and 
$$
T=\exp\(B(\log q)^{5/3}(\log\log q)^{1/3}\).
$$
Then for any primitive character $\chi$ modulo $q$
and any $s=\sigma+it$ with $\sigma>1-\eta_1$ and $|t|\le T$
we have
$$
L(s,\chi)\ll \eta_1^{-1},
$$
whereas for any $s=\sigma+it$ with $\sigma>1-\eta_2$ and $|t|>T$ we have
$$
L(s,\chi)\ll \eta_2^{-1},
$$
where the implied constants depend only on $\fq$.
\end{theorem}

\begin{proof}
In what follows, we write
\begin{equation}
\label{eq:M def} 
M=\exp\((\tfrac14\xi_0)^{-1/3}(\log q)^{2/3}(\log\ell)^{1/3}\)
\end{equation}
for some constant $\xi_0$, to be chosen later. We also recall the choice of the parameters~\eqref{eq:def tau l M}.
Adjusting the values of $\gamma_0$ and $\xi_0$ if necessary, we can assume that
$M\ge\fq^{\gamma_0}$. Consequently, if $\ell_0$
is determined via the relation
\begin{equation}
\label{eq:l0 def}
\frac{4\ell_0\log\fq}{\log q}=(\tfrac14\xi_0)^{-1/3}(\log q)^{2/3}(\log\ell_0)^{1/3},
\end{equation}
then we have
\begin{equation}
\label{eq:Q0 M}  
Q_0\le M\quad\Longleftrightarrow\quad\ell\le\ell_0
\end{equation}
and the asymptotic relation
\begin{equation}
\label{eq:stunning-props}
\ell_0\sim c_1(\log q)^{5/3}(\log\log q)^{1/3}
\end{equation}
holds with some absolute constant $c_1>0$. 

To meet the conditions~\eqref{eq:Q_0Xeta-conditions} in Lemma~\ref{lem:etaXY-LFunction}
we must have $X\ge\max\{M,Q_0\}$. However, we also want $X$ to be reasonably small
in order to derive a strong upper bound on $|L(s,\chi)|$.  We take
$$
X=\max\{M,Q_0^{c_2}\},
$$
where $c_2\ge 1$ is a constant that is needed for technical reasons.

Given that the quantity $\eta^{-1}X^\eta$ appears in the bound~\eqref{eq:Kaldi}
on $L(s,\chi)$, the optimal choice of $\eta$ (at least for small values of $\ell$)
is $\eta=(\log X)^{-1}$.  This idea leads us to define
$$
\eta=\begin{cases}
\displaystyle\frac{(\tfrac14\xi_0)^{1/3}}{(\log q)^{2/3}(\log\ell)^{1/3}}
&\quad\hbox{if $\ell\le\ell_0$},\\ \\
\displaystyle\frac{\log q}{4c_2\ell\log\fq}
&\quad\hbox{if $\ell>\ell_0$}.
\end{cases}
$$
In particular, we see from~\eqref{eq:Q0 M} that 
\begin{equation}
\label{eq:X eta}
X^\eta  \ll 1
\end{equation}

It is also convenient to observe that 
$$
\eta \asymp \min\left\{ \frac{1}{(\log q)^{2/3}(\log\ell)^{1/3}}, \frac{\log q}{ \ell} \right\}
$$
in view of~\eqref{eq:stunning-props}, and since
$\log(qT)\asymp\ell_0$ we see that $\eta$ is asymptotic
to the first [resp.~second] term in the above minimum
if $|t|\le T$ [resp.~$|t|>T$]. Thus, it suffices 
to establish that 
\begin{equation}
\label{eq:L eta}
L(s,\chi)\ll\eta^{-1}\qquad(\sigma>1-\eta).
\end{equation}

Note that for $\ell\le\ell_0$ we have
\begin{equation}
\label{eq:sleepover}
\eta\le\frac{(2\xi_0)^{1/3}(\log\ell)^{2/3}}{(\log q)^{2/3}}
=\frac{\xi_0(\log M)^2}{(\log q)^2}-\frac{2\log\ell}{\log M}
\le\frac{\xi_0(\log X)^2}{(\log q)^2}-\frac{2\log\ell}{\log X},
\end{equation}
and adjusting the value of $\gamma_0$ if necessary, we also have
$\eta<\tfrac13$; therefore, all of the conditions~\eqref{eq:Q_0Xeta-conditions}
are met. 

For $\ell>\ell_0$ we see that~\eqref{eq:sleepover} again holds as a 
consequence of~\eqref{eq:stunning-props}
provided that $c_2$ is sufficiently large.  We can also guarantee that
$\eta<\tfrac13$ by taking $c_2$ large enough; for example,
$c_2 > \tfrac{3}{4} \log 2$ suffices.
Hence, for a suitable choice of $c_2$ the conditions~\eqref{eq:Q_0Xeta-conditions}
are all met when $\ell>\ell_0$.

By Lemma~\ref{lem:etaXY-LFunction} we have
$$
L(s,\chi)\ll\eta^{-1}X^\eta+Y^\eta \ell q^{-c_0}\qquad(\sigma>1-\eta).
$$
where $Y$ is given by~\eqref{eq:def Y}.

Recalling~\eqref{eq:X eta}, we see that to establish~\eqref{eq:L eta}
it suffices to show that
\begin{equation}
\label{eq:I'm-cold}
Y^\eta \ell q^{-c_0}\ll\eta^{-1}
\end{equation}
with the choices we have made.
When $\eta\asymp(\log q)^{-2/3}(\log\ell)^{-1/3}$
it is easy to see that~\eqref{eq:I'm-cold} holds whenever
$\ell\ll(\log q)^{20/9}(\log\log q)^{5/9}$. Taking into account~\eqref{eq:stunning-props}, 
this establishes~\eqref{eq:I'm-cold} when $\ell\le\ell_0$.  

For larger $\ell$, note that
we can ensure that the inequality
$$
\frac{\log q}{4c_2\ell\log\fq}
\cdot 60(\ell\log 2\ell)^{3/4}\le c_0\log q-\log\log q
$$
by taking $c_2$ sufficiently large at the outset.  Exponentiating
both sides of this inequality, we obtain~\eqref{eq:I'm-cold} when $\ell>\ell_0$.
\end{proof} 

\section{The zero-free region}
\label{sec:the-zero-free-region}

We continue to use the notation of the
first paragraph of Section~\ref{sec:bounds-on-L-functions-I}, in particular
we use the parameters defined in~\eqref{eq:def tau l M} and~\eqref{eq:def Y}.

The following technical result originates with the work of Iwaniec~\cite{Iwan};
the specific statement given here is due to Banks and Shparlinski~\cite[Lemma~6.2]{BaSh}.

\begin{lemma}
\label{lem:technicallemma}
Let $q$ be a fixed modulus. Let $\eta\in(0,\frac13)$, $K\ge e$ and $T\ge 1$
be given numbers, which may depend on $q$. Put
$$
\vartheta=\frac{\eta}{400\log K},
$$
and suppose that
$$
8\log(5\log 3q)+\frac{24}{\eta}\log(2K/5\vartheta)\le\frac{1}{15\vartheta}.
$$
Suppose that $|L(s,\chi)|\le K$ for all primitive characters $\chi$
modulo~$q$ and all $s$ in the region
$\left \{s\in\CC:\sigma>1-\eta,~|t|\le 3T\right\}$. 
Then there is at most one primitive character $\chi$ modulo $q$
such that $L(s,\chi)$ has a zero in 
$\left\{s\in\CC:\sigma>1-\vartheta,~|t|\le T\right\}$.
If such a character exists, then it is a real character, and
the zero is unique, real and simple.
\end{lemma}

The main result of this section is the following.

\begin{theorem}
\label{thm:nonvanish-LFunction}
There are constants $A,B>0$ that depend only on $\fq$ and have
the following property. Put 
$$
\vartheta_1=\frac{A}{(\log q)^{2/3}(\log\log q)^{1/3}},\qquad
\vartheta_2=\frac{A\log q}{\ell},
$$
and  
$$
T=\exp\(B(\log q)^{5/3}(\log\log q)^{1/3}\).
$$
Then for any primitive character $\chi$ modulo $q$,
the Dirichlet $L$-function $L(s,\chi)$ does not vanish in the region
$$
\{s\in\CC:\sigma>1-\vartheta_1,~|t|\le T\}\cup
\{s\in\CC:\sigma>1-\vartheta_2,~|t|>T\}.
$$
\end{theorem}

\begin{proof}
The proof is similar to that of Theorem~\ref{thm:bound-on-LFunction},
but the parameters $\eta_1$ and $\eta_2$ are chosen here with the goal of producing the
largest possible zero-free region rather than minimizing an upper bound on $L(s,\chi)$.

In what follows, we can assume that
$$
Q_0=\fq^{4\ell/\log q},
$$
for otherwise $Q_0=\fq^{\gamma_0}$ and the result is already contained in
\cite[Theorem~3.2]{BaSh}. Let $Y$ be defined by~\eqref{eq:def Y} and as
in the proof of Theorem~\ref{thm:bound-on-LFunction} let $M$ and $\ell_0$
be defined by~\eqref{eq:M def} and~\eqref{eq:l0 def}, respectively.  Put
$$
\eta_1=\frac{(2\xi_0)^{1/3}(\log\ell)^{2/3}}{(\log q)^{2/3}}\mand
\eta_2=\frac{c_0\log q}{\log Y}=\frac{c_0\log q}{60(\ell\log 2\ell)^{3/4}}.
$$
One verifies that
\begin{equation}
\label{eq:sleepovery}
\eta_1=\frac{\xi_0(\log M)^2}{(\log q)^2}-\frac{2\log\ell}{\log M}.
\end{equation}
Let $\ell_1$ be determined via the relation
$$
\frac{(2\xi_0)^{1/3}(\log\ell_1)^{2/3}}{(\log q)^{2/3}}
\cdot 60(\ell_1\log 2\ell_1)^{3/4}=c_0\log q.
$$
Then we have
\begin{equation}
\label{eq:jiffy-lube}
Y^{\eta_1}\le q^{c_0}\quad\Longleftrightarrow\quad\ell\le\ell_1
\end{equation}
and the asymptotic relation
\begin{equation}
\label{eq:stunning-props2}
\ell_1\sim c_2(\log q)^{20/9}(\log\log q)^{-17/9}
\end{equation}
for some absolute constant $c_2>0$.  Clearly, $\ell_0<\ell_1$ for all large $q$.

To prove the result, we study three cases separately.

\bigskip{\underline{\sc Case 1}}: $\ell\le\ell_0$.
We put $X=M$ and $\eta=\eta_1$.  Using~\eqref{eq:sleepovery}
and the fact that $M\ge Q_0$ in this case,
we easily verify the conditions~\eqref{eq:Q_0Xeta-conditions} 
if $q$ is sufficiently large. By Lemma~\ref{lem:etaXY-LFunction}
we have
$$
L(s,\chi)\ll \eta^{-1}X^{\eta}+Y^{\eta}\ell q^{-c_0}\qquad(\sigma>1-\eta).
$$
By our choices of $X$ and $\eta$ it follows that $\eta^{-1}X^{\eta}\ll\ell^3$.
Moreover, $Y^\eta \ell q^{-c_0}\le\ell$ by~\eqref{eq:jiffy-lube}
since the asymptotic relation~\eqref{eq:stunning-props2} implies that
$\ell\le\ell_1$ for all large $q$.
Consequently,
$$
\log|L(s,\chi)|\le 3\log\ell+O(1)\ll\eta\log X\qquad(\sigma>1-\eta).
$$
Applying Lemma~\ref{lem:technicallemma}, taking into account that
$\log\ell\asymp\log\log q$ in view of~\eqref{eq:stunning-props},
we see that there are numbers $A_1,B>0$ with the following property.  Put
$$
\vartheta_1=\frac{A_1}{(\log q)^{2/3}(\log\log q)^{1/3}}\mand
T=\exp\(B(\log q)^{5/3}(\log\log q)^{1/3}\).
$$
Then there exists at most one primitive character $\chi$ modulo $q$
such that $L(s,\chi)$ has a zero in the region
$$
\cR_1=\{s\in\CC:\sigma>1-\vartheta_1,~|t|\le T\}.
$$
If such a character exists, then it is a real character, and
the zero is unique, real and simple.
To rule out the possibility of such an exceptional zero, we note that there
are at most $O(1)$ primitive real Dirichlet characters for which the core of the conductor is
the number $\fq$. Consequently, after replacing
$A_1$ with a smaller number depending only on $\fq$
(more precisely, on the locations of real zeros of these characters) we can
guarantee that $L(s,\chi)$ does not vanish in $\cR_1$ for any primitive
character modulo $q$.

\bigskip{\underline{\sc Case 2}}: $\ell_0<\ell\le\ell_1$. In this case we put
$X=Q_0^{c_3}$, where $c_3$ is a large positive constant, and we take
$\eta=\eta_1$ as before.  Since $\ell>\ell_0$ we have by~\eqref{eq:stunning-props}:
$$
\frac{\xi_0(\log X)^2}{(\log q)^2}-\frac{2\log\ell}{\log X}
>\frac{16\xi_0 c_3^2\ell_0^2(\log\fq)^2}{(\log q)^4}
-\frac{\log\ell_0\cdot\log q}{2c_3\ell_0\log\fq}
\sim\frac{c_4(\log\log q)^{2/3}}{(\log q)^{2/3}},
$$
where
$$
c_4=16\xi_0 c_1^2c_3^2(\log\fq)^2
-\frac{1}{2c_1c_3\log\fq}.
$$ 
Thus, if $c_3$ is large enough so that $c_4>(2\xi_0)^{1/3}$, then
the conditions~\eqref{eq:Q_0Xeta-conditions} are met for all large $q$.
By Lemma~\ref{lem:etaXY-LFunction} we again have
$$
L(s,\chi)\ll \eta^{-1}X^{\eta}+Y^{\eta}\ell q^{-c_0}\qquad(\sigma>1-\eta).
$$
Using our choices of $X$ and $\eta$, and taking into account that
$Y^\eta \ell q^{-c_0}\le\ell$ as in Case~1, it follows that
$$
\log|L(s,\chi)|\ll \frac{\ell(\log\ell)^{2/3}}{(\log q)^{5/3}}\qquad(\sigma>1-\eta),
$$
where the implied constant depends only on $\fq$.
Applying Lemma~\ref{lem:technicallemma}, we see that there are numbers
$A_2,B_2>0$ with the following property.  Put
$$
\vartheta_2=\frac{A_2\log q}{\ell}\mand
T_2=\exp\(B_2(\log q)^{20/9}(\log\log q)^{-17/9}\).
$$
Then for any primitive character $\chi$ modulo $q$, the $L$-function
$L(s,\chi)$ does not vanish in the region
$$
\cR_2=\{s\in\CC:\sigma>1-\vartheta_2,~T<|t|\le T_2\}.
$$

\bigskip{\underline{\sc Case 3}}: $\ell>\ell_1$. In this case we put
$X=Q_0^{c_3}$ as before, but now we set $\eta=\eta_2$.  As in Case~2 we have
\begin{equation}
\label{eq:crown}
\frac{\xi_0(\log X)^2}{(\log q)^2}-\frac{2\log\ell}{\log X}
\ge (c_4+o(1))\frac{(\log\log q)^{2/3}}{(\log q)^{2/3}}.
\end{equation}
In view of the asymptotic relation in~\eqref{eq:stunning-props2}, 
the left side of~\eqref{eq:crown} exceeds $\eta$ for all large $q$
provided that $c_4>c_0/(60c_2^{3/4})$, which we can guarantee by
choosing $c_3$ sufficiently large at the outset.
With these choices the conditions~\eqref{eq:Q_0Xeta-conditions} are met
for all large $q$. By Lemma~\ref{lem:etaXY-LFunction} we again have
$$
L(s,\chi)\ll \eta^{-1}X^{\eta}+Y^{\eta}\ell q^{-c_0}\qquad(\sigma>1-\eta).
$$
Using our choices of $X$ and $\eta$, and taking into account that
$Y^\eta\ell q^{-c_0}=\ell$ holds in view of the definition of
$\eta_2$, it follows that
$$
\log|L(s,\chi)|\ll
\frac{\ell^{1/4}}{(\log\ell)^{3/4}}\qquad(\sigma>1-\eta),
$$
where the implied constant depends only on $\fq$.
Applying Lemma~\ref{lem:technicallemma}, we see that
for any primitive character $\chi$ modulo $q$, the $L$-function
$L(s,\chi)$ does not vanish in the region
$$
\cR_3=\{s\in\CC:\sigma>1-\vartheta_2,~|t|>T_2\}.
$$

Combining the results of the three cases, and taking $A = \min\{A_1,A_2\}$,
we complete the proof.
\end{proof}

\begin{remark} The result of
Iwaniec~\cite[Theorem~2]{Iwan} is superior to 
Theorem~\ref{thm:nonvanish-LFunction} when $|t|\gg (\log q)^4(\log\log q)^3$.
Reducing $A>0$  in Theorem~\ref{thm:nonvanish-LFunction} so that
$$
\vartheta_3=\frac{A}{(\ell\log\ell)^{3/4}}
\le\frac{1}{40000(\log\fq+(\ell\log 2\ell)^{3/4})},
$$
and adjusting the constant $B$
if necessary, these results can be neatly combined:
For any character $\chi$ modulo $q$, the $L$-function
$L(s,\chi)$ does not vanish in the region
$\{s\in\CC:\sigma>1-\max\{\min\{\vartheta_1,\vartheta_2\},\vartheta_3\}\}$.
\end{remark}

\section{Non-vanishing and bounds on $L$-functions,
their logarithmic derivatives and reciprocals}
\label{sec:bounds-on-L-functions-II}

\begin{lemma}
\label{lem:Brady1}
Suppose $f(z)$ is analytic in a region that contains a disc $|z|\le\Delta$ with
$\Delta>0$, and that $f(0)\ne 0$.
Let $r$ and $R$ be real numbers such that $0<r<R<\Delta$.  Then 
$$
\left|\frac{f'}{f}(z)-\sum_{j=1}^J\frac{1}{z-z_j}\right|
\le b(\Delta,r,R)\log\frac{B}{|f(0)|}\qquad(|z|\le r),
$$
where the sum runs over all zeros $z_j$ of $f$ for which $|z_j|\le R$, and
$$
b(\Delta,r,R)=\frac{2R}{(R-r)^2}+\frac{1}{(R-r)\log(\Delta/R)}.
$$
\end{lemma}

\begin{proof}
This is a result of Montgomery and Vaughan~\cite[Lemma~6.3]{MontVau}
in the special case $\Delta=1$, hence we only sketch the proof;
the explicit form of the upper bound given here can be
obtained by keeping track of the constants that arise while
combining the Jensen inequality (see~\cite[Lemma~6.1]{MontVau}) with the
Borel-Carath\'eodory lemma (see~\cite[Lemma~6.2]{MontVau}).
The general case $\Delta>0$ is proved by applying 
the specific case $\Delta=1$ to the function given by
$F(w)=f(\Delta w)$ for all $w\in\CC$. Writing $z = \Delta w$
with $|w| \le 1$ for $|z|\le\Delta$ we have 
$$
\frac{f'}{f}(z) = \Delta^{-1} \frac{F'}{F}(w).
$$
Replacing $r\mapsto r\Delta$ and $R\mapsto R\Delta$, the general follows
from case $\Delta=1$ of the lemma applied to the function $F$; the details
are omitted.
\end{proof}

Let $q$ be a modulus satisfying~\eqref{eq:minmax},
$\chi$ a primitive character modulo~$q$, and put
$\tau=|t|+3$ and $\ell=\log(q\tau)$ as usual.
For the remainder of this section, all constants (including constants
implied by the symbols $\ll$, $\gg$, etc.) may depend on the core $\fq$
of $q$ but are otherwise absolute.

The next result combines our work in Sections~\ref{sec:bounds-on-L-functions-I}
and~\ref{sec:the-zero-free-region} with the main results of~\cite{Iwan}.  To
formulate the theorem, we introduce some notation. Let $A,B_1,B_2>0$ be fixed
real numbers, put
\begin{alignat*}{3}
&\eta_1=\frac{A}{(\log q)^{2/3}(\log\log q)^{1/3}},\qquad
&&\eta_2=\frac{A\log q}{\ell},\qquad
&&\eta_3=\frac{A}{\ell^{1/2}(\log\ell)^{3/4}},\\
&\vartheta_1=\tfrac12\eta_1,\qquad
&&\vartheta_2=\tfrac12\eta_2,\qquad
&&\vartheta_3= \tfrac12\ell^{-1/4}\eta_3,
\end{alignat*}
and denote
\begin{equation}
\begin{split}
\label{eq:T1T2-defn}
T_1&=\exp\(B_1(\log q)^{5/3}(\log\log q)^{1/3}\),\\
T_2&=\exp\(B_2(\log q)^4(\log\log q)^3\).
\end{split}
\end{equation}  
Define $\eta$ and $\vartheta$ by 
\begin{equation}
\label{eq: eta theta}
( \eta, \vartheta)=\begin{cases}
(\eta_1, \vartheta_1) &\quad\hbox{if $|t|\le T_1$},\\
(\eta_2, \vartheta_2) &\quad\hbox{if $T_1<|t|\le T_2$},\\
(\eta_3, \vartheta_3) &\quad\hbox{if $|t|>T_2$}. 
\end{cases}
\end{equation}
Finally, put
$$
K=\begin{cases}
(\log q)^{2/3}(\log\log q)^{1/3}&\quad\hbox{if $|t|\le T_1$},\\
\ell/\log q&\quad\hbox{if $T_1<|t|\le T_2$},\\
\exp(100\ell^{1/4})&\quad\hbox{if $|t|>T_2$}.
\end{cases}
$$
Then, combining Theorems~\ref{thm:bound-on-LFunction} and~\ref{thm:nonvanish-LFunction}
along with~\cite[Theorems~1 and~2]{Iwan}, we see that one can select the
constants $A,B_1,B_2>0$ so that the following holds.
 
\begin{theorem}
\label{eq:shrubbery!}
For any $s=\sigma+it$ with $\sigma\ge 1-\eta$ and  any primitive character $\chi$ modulo $q$ we have
$L(s,\chi)\ll K$, and  $L(s,\chi)$
does not vanish in the region
$\{s\in\CC:\sigma>1-\vartheta\}$.
\end{theorem}

For technical reasons (see Remark~\ref{rem:blahblahblah} below) we now define
\begin{equation}
\label{eq:loudmouth-T3-defn}
T_3=\exp\(B_2(\log q)^4(\log\log q)^{-1}\).
\end{equation}
Note that $T_3<T_2$, hence all of the preceding results hold if $T_2$ is
replaced by $T_3$ since the results of Iwaniec~\cite{Iwan} hold generally
for all $t\in\RR$. Although this modification would lead to a slightly weaker
form of Theorem~\ref{eq:shrubbery!}, it yields a slightly stronger (and more
convenient) form of Theorem~\ref{eq:curious-beasts} below.

We apply Lemma~\ref{lem:Brady1} with the choices
$$
f(z)=L(z+1+2\eta+it,\chi),\qquad
\Delta=3\eta,\qquad
r=2\eta+\vartheta\mand R=2.9\eta.
$$
Observe that zeros of $f$ are of the form $\rho - (1+2\eta+it)$ 
where $\rho$ runs through the zeros of  $L(s,\chi)$.

By Theorem~\ref{eq:shrubbery!} we can take $B=CK$ with a
suitable constant $C$, and using the Euler product we have
\begin{align*}
|f(0)|
=\prod_p\left|1-\chi(p)p^{-1-2\eta-it}\right|^{-1}
\ge\prod_p\left|1+p^{-1-2\eta}\right|^{-1}
=\frac{\zeta(2+4\eta)}{\zeta(1+2\eta)}\gg \eta;
\end{align*}
consequently,
$$
\log\frac{B}{|f(0)|}\ll\begin{cases}
\log\log q&\quad\hbox{if $|t|\le T_3$},\\
\ell^{1/4}&\quad\hbox{if $|t|>T_3$}.
\end{cases}
$$
Since $R-r=0.9\eta-\vartheta\asymp\eta$ and $\log(\Delta/R)\asymp 1$,
it follows that $b(\Delta,r,R)\asymp\eta^{-1}$.
Then Lemma~\ref{lem:Brady1} shows that for any $s=\sigma+it$
with $\sigma\in[1-\vartheta,1+4\eta+\vartheta]$ we have
\begin{equation}
\label{eq:L'L}
\frac{L'}{L}(s,\chi)=\sum_{\rho\in\cZ(t)}\frac{1}{s-\rho}+O(\Theta),
\end{equation}
where the sum is over the (possibly empty) set $\cZ(t)$ comprised of all
zeros $\rho$ of $L(s,\chi)$ for which $|\rho-(1+2\eta+it)|\le R=2.9\eta$, and
\begin{equation}
\label{eq:Theta}
\Theta=\begin{cases}
(\log q)^{2/3}(\log\log q)^{4/3}&\quad\hbox{if $|t|\le T_1$},\\
(\ell\log\log q)/\log q&\quad\hbox{if $T_1<|t|\le T_3$},\\
(\ell\log\ell)^{3/4}&\quad\hbox{if $|t|>T_3$}.
\end{cases}
\end{equation}
Note that the implied constant in \eqref{eq:L'L} is absolute.

\begin{remark}
\label{rem:blahblahblah}
Our motivation for introducing the parameter $T_3$ defined by~\eqref{eq:loudmouth-T3-defn}
is to account for the fact that $\Theta$, regarded as a function of $t$,
has an order of magnitude transition at $|t|\asymp T_3$ (and not at $|t|\asymp T_2$).
\end{remark}

\begin{theorem}
\label{eq:curious-beasts}
For any $s=\sigma+it$ with $\sigma>1-\frac12\Theta^{-1}$ for $L(s,\chi)$ with a primitive character $\chi$ of conductor $q$  we have the following bounds 
on the logarithmic derivative
\begin{equation}
\label{eq:Brady4a}
\frac{L'}{L}(s,\chi)\ll \Theta,
\end{equation}
its size
\begin{equation}
\label{eq:Brady4b}
|\log L(s,\chi)|\le\tfrac34\log \Theta+O(1),
\end{equation}
and  its reciprocal
\begin{equation}
\label{eq:Brady4c}
\frac{1}{L(s,\chi)}\ll \Theta.
\end{equation}
\end{theorem}

\begin{proof}
We follow the proof of~\cite[Theorem~11.4]{MontVau} closely,
making only minor modifications. Since
$$
\left|\frac{L'}{L}(s,\chi)\right|
\le -\frac{\zeta'}{\zeta}(\sigma)\ll\frac{1}{\sigma-1},
$$
the bound~\eqref{eq:Brady4a} is immediate when $\sigma\ge 1+\Theta^{-1}$.
Continuing the argument in that proof, here we set $s_*=1+\Theta^{-1}+it$
and in the same manner derive the upper bound (cf.~\cite[Equation~(11.11)]{MontVau})
\begin{equation}
\label{eq:coldplay1}
\sum_{\rho\in\cZ(t)}\frac{1}{s_*-\rho}
\ll \Theta
\end{equation}
and the estimate (cf.~\cite[Equation~(11.12)]{MontVau})
\begin{equation}
\label{eq:coldplay2}
\sum_{\rho\in\cZ(t)}\frac{1}{s-\rho}
=\sum_{\rho\in\cZ(t)}\(\frac{1}{s-\rho}-\frac{1}{s_*-\rho}\)+O(\Theta),
\end{equation}
where again $\cZ(t)$ is the set of zeros $\rho$ of $L(s,\chi)$
for which $|\rho-(1+2\eta+it)|\le 2.9\eta$.

Now suppose that $1-\tfrac12\Theta^{-1}\le\sigma\le 1+\Theta^{-1}$.
If $\rho=\beta+i\gamma\in\cZ(t)$, then
by the definition of $\cZ(t)$ we have
$$
|\beta - (1+2\eta)| = |\Re \rho-(1+2\eta+it)|\le 2.9\eta;
$$
hence $\beta \ge 1-0.9\eta$. In view of Theorem~\ref{eq:shrubbery!} 
we see that $\beta\in[1-0.9\eta,1-\vartheta]$.
Consequently,
$$
1\ge\frac{\Re(s-\rho)}{\Re(s_*-\rho)}
\ge\frac{1-\tfrac12\Theta^{-1}-\beta}{1+\Theta^{-1}-\beta}
\ge \frac{0.9\eta-\tfrac12\Theta^{-1}}{0.9\eta+\Theta^{-1}}\gg 1.
$$
Since $\Im(s-\rho)=\Im(s_*-\rho)$, it follows that
$|s-\rho|\asymp|s_*-\rho|$ for every $\rho\in\cZ(t)$.
Consequently, for all zeros $\rho\in\cZ(t)$
we have for $\sigma\ge 1-\tfrac12\Theta^{-1}$:
$$
\frac{1}{s-\rho}-\frac{1}{s_*-\rho}
=\frac{s_*-s}{(s-\rho)(s_*-\rho)}
=\frac{1+\Theta^{-1}-\sigma}{(s-\rho)(s_*-\rho)}
\ll\frac{\Theta^{-1}}{|s_*-\rho|^2}
\le\Re\frac{1}{s_*-\rho}.
$$
Summing this over $\rho\in\cZ(t)$, taking into 
account~\eqref{eq:coldplay1} and~\eqref{eq:coldplay2}, we deduce the bound
$$
\sum_{\rho\in\cZ(t)}\frac{1}{s-\rho}\ll \Theta.
$$
In view of~\eqref{eq:L'L} this completes the proof of~\eqref{eq:Brady4a},
which in turn yields the bounds~\eqref{eq:Brady4b} and~\eqref{eq:Brady4c}
via the same argument given in the proof of~\cite[Theorem~11.4]{MontVau}, making
use of the bound
$$
\log L(s,\chi)-\log L(s_*,\chi)=\int_{s_*}^s\frac{L'}{L}(u,\chi)\,du
\ll |s_*-s|\Theta\ll 1
$$
in the case that  $1-\tfrac12\Theta^{-1}\le\sigma\le 1+\Theta^{-1}$.
\end{proof}

\section{Proof of Theorem~\ref{thm:mu-Lambda-sums}}
\label{sec:proof main}

We continue to use the notation of Section~\ref{sec:bounds-on-L-functions-II}.
In particular, all constants (including constants
implied by the symbols $\ll$, $\gg$, etc.) may depend on $\fq$
but are otherwise absolute.

To bound the sums $M(x,\chi)$ and $\psi(x,\chi)$
defined by~\eqref{eq:M chi x defn} and~\eqref{eq:psi tau chi x defn},
we use Perron's formula in conjunction with
the bounds provided by Theorem~\ref{eq:curious-beasts}.  The techniques we use are
standard and well known; we follow an approach outlined in the
book~\cite{MontVau} of Montgomery and Vaughan.

As in the proof of~\cite[Theorem~11.16]{MontVau}, we start from the estimates
\begin{align}
\label{eq:yak1}
\psi(x,\chi)&=\frac{-1}{2\pi i}\int_{\sigma_0-iT}^{\sigma_0+iT}
\frac{L'}{L}(s,\chi)\frac{x^s}{s}\,ds+R_\Lambda,\\
\label{eq:yak2}
M(x,\chi)&=\frac{1}{2\pi i}\int_{\sigma_0-iT}^{\sigma_0+iT}
\frac{1}{L(s,\chi)}\frac{x^s}{s}\,ds+R_\mu,
\end{align}
where $\sigma_0=1+(\log x)^{-1}$, $T$ is a parameter in $[2,x]$ to
be specified below, and for $f=\Lambda$ or $\mu$ the error term $R_f$
satisfies the bound
$$
R_f\ll\begin{cases}
x(\log x)^2T^{-1}&\quad\hbox{if $f=\Lambda$},\\
x(\log x)T^{-1}&\quad\hbox{if $f=\mu$}.
\end{cases}
$$
The integrals in~\eqref{eq:yak1} and~\eqref{eq:yak2} can be handled
with the same argument using~\eqref{eq:Brady4a} and~\eqref{eq:Brady4c},
respectively.  For this reason, here we consider only $f=\Lambda$.

Before proceeding, recall that the parameters $\vartheta$ and $\Theta$ are
functions of the real variable $t$ (up to now, the dependence
on $t$ has been suppressed in the notation for the sake of simplicity), and thus
we write them as  $\vartheta(t)$ and $\Theta(t)$.
In particular, recalling that by~\eqref{eq:T1T2-defn}
and~\eqref{eq:loudmouth-T3-defn} we have
$$
T_1=\exp\(B_1(\log q)^{5/3}(\log\log q)^{1/3}\)\quad\text{and}\quad
T_3=\exp\(B_2(\log q)^4(\log\log q)^{-1}\),
$$
we see from~\eqref{eq:Theta} that
$$
\Theta(t)\asymp\begin{cases}
(\log q)^{2/3}(\log\log q)^{4/3}&\quad\hbox{if $2\le |t|\le T_1$},\\
(\log |t|\cdot\log\log q)/\log q&\quad\hbox{if $T_1<|t|\le T_3$},\\
(\log |t|\cdot\log\log |t|)^{3/4}&\quad\hbox{if $|t|>T_3$}.
\end{cases}
$$
Here, we have used the fact that $\ell=\log(q\tau)\asymp\log t$ whenever $|t|>T_1$.
Similarly, we see from~\eqref{eq: eta theta} that
$$
\vartheta(t)\asymp\begin{cases}
(\log q)^{-2/3}(\log\log q)^{-1/3}
&\quad\hbox{if $2\le |t|\le T_1$},\\
(\log |t|)^{-1}\log q
&\quad\hbox{if $T_1<|t|\le T_2$},\\
(\log |t|\cdot\log\log |t|)^{-3/4}
&\quad\hbox{if $|t|>T_2$}.
\end{cases}
$$
Consequently, there is a constant
$c\in(0,\frac12)$ depending only on $\fq$ for which
\begin{equation}
\label{eq:wired-for-sound}
c\,\Theta(T)^{-1}<\vartheta(t)\qquad(|t|\le T).
\end{equation}
Finally, notice that
\begin{equation}
\label{eq:wolf-blitzer}
\log x\gg(\log q)^{2/3}(\log\log q)^{4/3}\quad\Longrightarrow\quad
\Theta(T)\ll\log x
\end{equation}
and we see below that the latter bound is needed in order to derive nontrivial
bounds on the sums we are considering.

Following the proof of~\cite[Theorem~6.9]{MontVau} we denote
by $\cC$ a closed contour that consists of line segments connecting the points
$\sigma_0-iT$, $\sigma_0+iT$, $\sigma_1+iT$ and $\sigma_1-iT$, where
$\sigma_1=1-\eps\,\Theta(T)^{-1}$.  From~\eqref{eq:wired-for-sound}
and Theorem~\ref{eq:shrubbery!} it follows that $L(s,\chi)$ does not vanish
inside the contour $\cC$; therefore,
$$
0=\oint\limits_{\cC}\frac{L'}{L}(s,\chi)\frac{x^s}{s}\,ds
=\(\int_{\sigma_0-iT}^{\sigma_0+iT}
+\int_{\sigma_0+iT}^{\sigma_1+iT}
+\int_{\sigma_1+iT}^{\sigma_1-iT}
+\int_{\sigma_1-iT}^{\sigma_0-iT}\)\frac{L'}{L}(s,\chi)\frac{x^s}{s}\,ds.
$$
By~\eqref{eq:yak1} and the bound on $R_\Lambda$ we have
$$
\int_{\sigma_0-iT}^{\sigma_0+iT}
\frac{L'}{L}(s,\chi)\frac{x^s}{s}\,ds
=-2\pi i\psi(x,\chi)+O(x(\log x)^2T^{-1})
$$
Next, using~\eqref{eq:Brady4a} and the bound $\Theta(T)\ll\log x$ we have
$$
\int_{\sigma_1\pm iT}^{\sigma_0\pm iT}
\frac{L'}{L}(s,\chi)\frac{x^s}{s}\,ds
\ll\frac{\Theta(T)}{T}\int_{\sigma_1}^{\sigma_0}x^\sigma\,d\sigma 
\le\frac{\Theta(T)}{T}\frac{x^{\sigma_0}}{\log x}
\ll\frac{x}{T}.
$$
Finally, using~\eqref{eq:Brady4a} and $\Theta(T)\ll\log x$ again we see that
$$
\int_{\sigma_1-iT}^{\sigma_1+iT}\frac{L'}{L}(s,\chi)\frac{x^s}{s}\,ds
\ll x^{\sigma_1}\int_{-T}^T\frac{\Theta(t)}{|t|+1}\,dt
\ll x^{\sigma_1}\Theta(T)\log T
\ll x^{\sigma_1}(\log x)^2.
$$
Putting everything together, and recalling the definition of $\sigma_1$,
we deduce that
\begin{equation}
\label{eq:psi-cold-brew}
\psi(x,\chi)
\ll x(\log x)^2\(\frac{1}{T}+\exp\(-\frac{c\log x}{\Theta(T)}\)\).
\end{equation}
By a similar argument we have
$$
M(x,\chi)
\ll x\log x\(\frac{1}{T}+\exp\(-\frac{c\log x}{\Theta(T)}\)\).
$$
Note that these bounds are trivial unless $\Theta(T)\ll\log x$, which
is the reason we assume~\eqref{eq:obama} (cf.~\eqref{eq:wolf-blitzer}).

It now remains to optimise $T$.  
Recall that the definitions of $T_1$ and $T_3$ in ~\eqref{eq:T1T2-defn} and~\eqref{eq:loudmouth-T3-defn} again.
The constants $B_1,B_2>0$ depend only on $\fq$, and it is clear from our methods
that these numbers can be chosen (and fixed) with $B_1\le B_2$. 

To optimize the bounds in~\eqref{eq:psi-cold-brew}
our aim is to select $T$ so that $\Theta(T)\log T\asymp\log x$; this requires
only a drop of care.

\bigskip{\underline{\sc Case 1}}: $2\le T\le T_1$. In this range
$\Theta(T)\asymp(\log q)^{2/3}(\log\log q)^{4/3}$, so we put
$$
T=\exp\(\frac{B_1\log x}{(\log q)^{2/3}(\log\log q)^{4/3}}\).
$$
With this choice, the condition $2\le T\le T_1$ is verified if
\begin{equation}
\label{eq:BadSanta}
B_1^{-1}(\log q)^{2/3}(\log\log q)^{4/3}\le\log x\le (\log q)^{7/3}(\log\log q)^{5/3}.
\end{equation}
As the results of Theorem~\ref{thm:mu-Lambda-sums}
are trivial when $\log x < B_1^{-1}(\log q)^{2/3}(\log\log q)^{4/3}$,
we can dispense with the first inequality in~\eqref{eq:BadSanta}, which then simplifies
as the condition $x \le Q_1$.

\bigskip{\underline{\sc Case 2}}: $T_1<T\le T_3$. In this range
$\Theta(T)\asymp(\log T\cdot\log\log q)/\log q$, and to optimize we put
$$
T=\exp\(\frac{B_1(\log x\cdot\log q)^{1/2}}{(\log\log q)^{1/2}}\).
$$
With this choice, the condition $T_1<T\le T_3$ is verified if
$$
(\log q)^{7/3}(\log\log q)^{5/3}<\log x
\le (B_2/B_1)^2(\log q)^7(\log\log q)^{-1}.
$$
Since $B_1\le B_2$ this includes the range $ Q_1 < x \le Q_2$,
so we are done in this case.

\bigskip{\underline{\sc Case 3}}: $T>T_3$. In this range
$\Theta(T)\asymp (\log T\cdot\log\log T)^{3/4}$ so we set
$$
T=\exp\(\frac{16B_2(\log x)^{4/7}}{(\log\log x)^{3/7}}\). 
$$
With this choice, the condition $T>T_3$ is verified if
\begin{equation}
\label{eq:trump-sux}
\frac{(\log x)^{4/7}}{(\log\log x)^{3/7}}
>\frac{(\log q)^4}{16\log\log q}.
\end{equation}
However, since $q\ge 3$ and $x>Q_2$ (that is, for $\log x>(\log q)^7(\log\log q)^{-1}$) 
it follows that
$$
\frac{\log x}{(\log\log x)^{3/4}}
>\frac{(\log q)^7(\log\log q)^{-1}}{(7\log\log q-\log\log\log q)^{3/4}}
>\frac{(\log q)^7}{128(\log\log q)^{7/4}},
$$
which implies~\eqref{eq:trump-sux} and therefore finishes the proof.

\section*{Acknowledgements}

W.~D.~Banks was supported in part by a grant from the University of
Missouri Research Board.
I.~E.~Shparlinski was supported in part by Australian Research Council
  Grant~DP170100786

%
%
%
%
%
%


\begin{thebibliography}{99}


\bibitem{BaSh} W.~D.~Banks  and I.~E.~Shparlinski,
``Bounds on short character sums and $L$-functions 
 for characters with a smooth modulus.''
\emph{J.\ d'Analyse Math.},  to appear (available
from \url{http://arxiv.org/abs/1605.07553}).

\bibitem{Bour} J.~Bourgain, 
``M{\"o}bius--Walsh correlation bounds and an estimate of Mauduit and Rivat.''
\emph{J.\ d'Analyse Math.} 119 (2013), 147--163. 

\bibitem{Chowla} S.~Chowla,
``The Riemann hypothesis and Hilbert's tenth problem.''
\emph{Mathematics and its applications}, 4, Gordon and Breach,
New York, London, Paris, 1965.

\bibitem{Gal} P.~X.~Gallagher,
``Primes in progressions to prime-power modulus.''
\emph{Invent.\ Math.} 16 (1972), 191--201.


\bibitem{Green} B.~Green,
``On (not) computing the M\"obius function using bounded depth circuits.''
\emph{Combin.\ Prob.\ Comp.} 21 (2012), 942--951. 

  


\bibitem{HK}
G.~Harman and I.~K\'atai,
``Primes with preassigned digits II.''
\emph{Acta Arith.} 133 (2008), 171--184.

\bibitem{Hinz}
J.~Hinz,
``Eine Erweiterung des nullstellenfreien Bereiches der Heckeschen Zetafunktion
und Primideale in Idealklassen.''
\emph{Acta Arith.} 38 (1980/81),  209--254. 

\bibitem{Iwan} H.~Iwaniec,
``On zeros of Dirichlet's $L$ series.''
\emph{Invent.\ Math.} 23 (1974), 97--104.

\bibitem{IwKow} H.~Iwaniec and E.~Kowalski,
\emph{Analytic number theory.} Amer.\  Math.\  Soc.,
Providence, RI, 2004.

\bibitem{LMN}
N.~Linial, Y.~Mansour and N.~Nisan, 
``Constant depth circuits, Fourier transform, and learnability.''
\emph{J.\ Assoc.\ Comput.\ Mach.} 40 (1993), 607--620.

\bibitem{MontVau}
H.~L.~Montgomery and R.~C.~Vaughan,
\emph{Multiplicative number theory~I. Classical theory}.
Cambridge Studies in Advanced Mathematics, 97.
Cambridge University Press, Cambridge, 2007.

\bibitem{Sarn}
P.~Sarnak, 
``M{\"o}bius  randomness and dynamics.''
\emph{Not.\ South Afr.\ Math.\ Soc.} 43 (2012), 89--97.



\bibitem{Sound}
K.~Soundararajan,
``Partial sums of the M\"{o}bius function.''
\emph{J.\ Reine Angew. Math.} 631 (2009), 141--152.

\bibitem{Wal}
A.~Walfisz,
\emph{Weylsche Exponentialsummen in der neueren Zahlentheorie},
Leipzig: B.G.~Teubner, 1963.

\end{thebibliography}
\end{document}